\documentclass[a4paper, reqno, 11pt]{amsart}

\usepackage{mathrsfs}
\usepackage{amsmath, amscd, amsthm,amssymb, amsfonts, verbatim,subfigure, enumerate}
\usepackage[mathcal]{eucal}
\usepackage[parfill]{parskip}
\usepackage[latin1]{inputenc}
\usepackage{graphicx}
\usepackage{enumitem}
\usepackage{color, a4wide}
\usepackage{tikz-cd}
\usepackage{tikz}
\usetikzlibrary{matrix,arrows}
\usetikzlibrary{patterns}
\renewcommand{\d}{\mathrm{d}}
\renewcommand{\L}{\mathcal{L}}

\usepackage{mathabx}

\newcommand{\what}[1]{\widehat{#1}}
\renewcommand{\d}{\mathrm{d}}
\newcommand{\dbar}{\bar{\partial}}
\newcommand{\eps}{\varepsilon}
\newcommand{\op}{\operatorname}
\newcommand{\CO}{{\mathcal{O}}}
\newcommand{\CL}{\mathcal{L}}
\newcommand{\C}{{\mathbb{C}}}
\newcommand{\R}{{\mathbb{R}}}
\newcommand{\Z}{{\mathbb{Z}}}
\newcommand{\E}{\mathcal{E}}
\newcommand{\CN}{\mathcal{N}}
\newcommand{\CM}{\mathcal{M}}
\newcommand{\g}{\mathfrak{g}}
\newcommand{\h}{\mathfrak{h}}

\newcommand{\F}{\mathcal{F}}
\DeclareMathOperator{\GL}{GL}
\DeclareMathOperator{\Spin}{Spin}
\DeclareMathOperator{\SU}{SU}
\DeclareMathOperator{\Tr}{Tr}
\DeclareMathOperator{\Obs}{Obs}

\DeclareMathOperator{\Sym}{Sym}
\newcommand{\ad}{\mathrm{ad}}
\newcommand{\cinfty}{C^{\infty}}
\newcommand{\id}{\operatorname{id}}
\newcommand*{\longhookrightarrow}{\ensuremath{\lhook\joinrel\relbar\joinrel\rightarrow}}

\newcommand*\strip{
\begin{tikzpicture}[scale=0.6]
\draw (0,0) -- (1, 0);
\draw (0,0.4) -- (1, 0.4);
\fill[pattern=north east lines] (0,0) rectangle (1,0.4);
\end{tikzpicture}
}
\newtheorem{thm}{Theorem}[section]
\newtheorem*{thm*}{Theorem}
\newtheorem{prop}[thm]{Proposition}
\newtheorem*{prop*}{Proposition}

\newtheorem{lemma}[thm]{Lemma}

\newtheorem{conjecture}[thm]{Conjecture}

\theoremstyle{definition}
\newtheorem{defn}[thm]{Definition}
\newtheorem*{defn*}{Definition}
\newtheorem{ex}[thm]{Example}
\newtheorem{rem}[thm]{Remark}

\theoremstyle{rem}
\newtheorem*{remark}{Remark}

\title{Lectures on mathematical aspects of (twisted) supersymmetric gauge theories}

\author{Kevin Costello}
\thanks{Kevin Costello
\newline Department of Mathematics,
Northwestern University, 2033 Sheridan Road,
Evanston, IL, 60208-2730, USA
{\tt costello@math.northwestern.edu}}

\author{Claudia Scheimbauer}
\thanks{Claudia Scheimbauer
\newline Department of Mathematics,
ETH Z\"urich, R\"amistrasse 101,
CH-8092 Z\"urich,
Switzerland,
{\tt scheimbauer@math.ethz.ch}
}

\begin{document}

\begin{abstract}

Supersymmetric gauge theories have played a central role in applications of quantum field theory to mathematics. Topologically twisted supersymmetric gauge theories often admit a rigorous mathematical description: for example, the Donaldson invariants of a 4-manifold can be interpreted as the correlation functions of a topologically twisted $\CN=2$ gauge theory. 
The aim of these lectures is to describe a mathematical formulation of partially-twisted supersymmetric gauge theories (in perturbation theory). These partially twisted theories are intermediate in complexity between the physical theory and the topologically twisted theories. Moreover, we will sketch how the operators of such a theory form a two complex dimensional analog of a vertex algebra. Finally, we will consider a deformation of the $\CN=1$ theory and discuss its relation to the Yangian, as explained in \cite{Cos13} and \cite{Cos13b}. 
\end{abstract}

\maketitle

\vspace{1em}
These are lecture notes of a minicourse given by the first author at the Winter school in Mathematical Physics 2012 in Les Houches on minimal (or holomorphic) twists of supersymmetric gauge theories, and will appear in the proceedings of this conference. 

Supersymmetric gauge theories in general are very difficult to study, whereas topologically twisted supersymmetric gauge theories have been well-studied. Our object of interest lies somewhere in between:

\begin{center}
{Supersymmetric gauge theories}\\\vspace{0.5em}
\rotatebox[origin=c]{90}{$\subseteq$}\\\vspace{0.5em}
\fbox{\bf Minimal (or holomorphic) twists of supersymmetric gauge theories}\\\vspace{0.5em}
\rotatebox[origin=c]{90}{$\subseteq$}\\\vspace{0.5em}
Topologically twisted supersymmetric gauge theories (e.g.\ Donaldson theory)
\end{center}

In the first section, we recall basics of supersymmetry. We define and describe holomorphic twists of $\CN=1,2,4$ supersymmetric theories. In the second section, we discuss the structure of the observables of field theories. In the last section, we examine the structure of the observables of twisted SUSY gauge theories more closely and explain a relation to vertex algebras. Moreover, we consider a deformation of the $\CN=1$ theory and discuss its relation to the Yangian and (conjecturally) to the quantum loop algebra. In a short appendix, we briefly summarize the framework set up in \cite{Cos11b} relating perturbative field theories, moduli problems, and elliptic $L_\infty$-algebras.

{\em Acknowledgements:} We are grateful to R.~Grady for his careful reading of the paper and useful comments. Moreover, the second author is very thankful to Damien Calaque for many helpful conversations. K.C.~is partially supported by NSF grant DMS 1007168, by a Sloan Fellowship, and by a Simons Fellowship in Mathematics. C.S.~is supported by a grant from the Swiss National Science Foundation (project number 200021\textunderscore137778). 


\section{Basics of Supersymmetry}\label{SUSY}
In these lectures, we consider gauge theories on $\R^4$. Everything in this first section is essentially standard, a reference for this material is \cite{Cos11b,DelFre99,Fre99}.

\subsection{Super-translation Lie algebra and supersymmetric field theories}

Recall that there is an isomorphism of groups
$$\Spin(4)\cong\SU(2)\times\SU(2).$$
Let $S^+$ and $S^-$ be the fundamental representations of the two $\SU(2)$'s. More precisely, referring to the two copies of $\SU(2)$ in $\Spin(4)$ as $\SU(2)_{\pm}$, let $S^+$ be the 2-dimensional complex fundamental representation of $\SU(2)_+$ endowed with trivial $\SU(2)_-$ action. Thus, $S^+$ is a 2-dimensional complex representation of $\Spin(4)$, and similarly, so is $S^-$.

Let $V_{\R}=\R^4$ and $V_{\C}=V_{\R}\otimes\C$.
Then $V_{\R}=\R^4$ is the defining 4-dimensional real representation of $\mathrm{SO}(4)$ and
$$V_{\C}\cong S^+\otimes S^-$$
as complex $\Spin(4)$ representations.

\begin{defn}
The {\em super-translation Lie algebra} $T^{\CN=1}$ is the complex $\Z/2\Z$-graded Lie algebra\footnote{Here $\Pi C$ means that the vector space $C$ has odd degree. So $T^{\CN=1}$ consists of $V_{\C}$ in degree 0 and $S^+\oplus S^-$ in degree 1.}
$$T^{\CN=1}=V_{\C}\oplus \Pi(S^+\oplus S^-),$$
where the Lie bracket is defined by $[Q^+, Q^-]=Q^+\otimes Q^-\in V_{\C}$ for $Q^+\in S^+, Q^-\in S^-$, and is zero otherwise.
\end{defn}

To encode more supersymmetry, we extend this definition to the following.
\begin{defn}
Let $W$ be a complex vector space. Define the {\em super-translation Lie algebra} $T^{W}$ to be the complex $\Z/2\Z$-graded Lie algebra
$$T^{W}=V_{\C}\oplus \Pi(S^+\otimes W\oplus S^-\otimes W^*),$$
where the Lie bracket is defined by $[Q^+\otimes w, Q^-\otimes w^*]=(Q^+\otimes Q^-)\langle w,w^*\rangle\in V_{\C}$ for $Q^+\otimes w\in S^+\otimes W,$ and $Q^-\otimes w^*\in S^+\otimes W^*$. The {\em number of supersymmetries} is the dimension of $W$. For $W=\C^k$ we use the notation
$$T^{\CN=k}=T^{\C^k}.$$
\end{defn}
Note that $\Spin(4)$ acts on $T^W$. 
\begin{defn}
A {\em supersymmetric (SUSY) field theory\footnote{For simplicity, we omit formal definitions here. See the appendix or \cite{Cos11b} for more details.} on $\R^4$} is a field theory on $\R^4 = V_\R$, equivariant under the action of $\op{Spin}(4) \ltimes V_\R$ on $\R^4$, and where the action of the Lie algebra $V_{\R}$ of translations is extended to an action of the Lie algebra of super-translations, in a way compatible with the $\op{Spin}(4)$-action. 
\end{defn}

Observe that $\GL(W)$ acts on $T^W$ naturally. If $G_R\subseteq\GL(W)$, one can ask that a SUSY field theory has a compatible action of $G_R$ and $T^W$.  In physics parlance, $G_R$ is the {\em R-symmetry group} of the theory.

\subsection{Twisting}

The general yoga of deformation theory \cite{KonSoi,Hin01,Lur11} tells us that symmetries\footnote{In order for this relationship to be a bijection, the word ``symmetry'' needs to be understood homotopically: e.g. by considering symmetries of a free resolution of an algebraic object.} of any mathematical object of cohomological degree 1 correspond to first order deformations.   More generally, symmetries of degree $k$ give first-order deformations over the base ring $\C[\eps]/\eps^2$, where $\eps$ is of degree $1-k$.  
The idea is the following.  Suppose we're dealing with a differential-graded mathematical object, such as a differential graded algebra $A$ with differential $d$.  A symmetry of $A$ of degree $k$ is a derivation $X$ of $A$ of degree $k$.  The corresponding deformation is given by changing the differential to $\d + \eps X$, where as above $\eps$ has degree $1-k$ and we work modulo $\eps^2$.  

Suppose that we have a supersymmetric field theory, acted on by the supersymmetry Lie algebra $T^W$.  Let us pick an odd element $Q \in T^W$.  In the supersymmetric world, things are bi-graded, by $\Z$ and $\Z/2$. We have both a cohomological degree and a ``super'' degree.  The symmetry $Q$ of our theory is of bidegree $(0,1)$; i.e. it is of cohomological degree $0$ and super degree $1$.  Thus, $Q$ will define a deformation of this theory over the base ring $\C[t]/t^2$, where the parameter $t$ is of bidegree $(1,1)$ (and thus even). 

Concretely, this deformation of our theory is obtained by adding $t Q$ to the BRST differential of the theory.  For example, if the theory is described by a factorization algebra (as we will discuss later), we're adding $t Q$ to the differential of the factorization algebra.

In general, first order deformations (corresponding to symmetries of degree $1$) extend to all-order deformations if they satisfy the Maurer-Cartan equation. In the example of a differential graded algebra described above, a derivation $X$ of $A$ of degree $1$ satisfies the Maurer-Cartan equation if 
$$
d X + \tfrac{1}{2} [X,X] = 0.
$$
This implies that the differential $d + \eps X$ has square zero, where we're working over the base ring $\C[\![\eps]\!]$.

The Lie algebra $T^W$ has zero differential, so that the Maurer-Cartan equation for an odd element $Q \in T^W$ is the equation $[Q,Q] = 0$.  Therefore, if $Q$ satisfies this equation, then it gives rise to a deformation of our theory over the base ring $\C[\![t]\!]$, where again $t$ is of bidegree $(1,1)$.  The twisted theory will be constructed from this deformation.
 
However, now we see that there's a problem: we would like our twisted theory to be a single $\Z \times \Z/2$-graded theory, not a family of theories over $\C[\![t]\!]$ where $t$ has bidegree $(1,1)$.  (The fact that $t$ has this bidegree means that, even if we could set $t = 1$, the resulting theory would not be $\Z \times \Z/2$-graded.)

To resolve this difficulty, we use a $\C^\times$ action to change the grading.
\begin{defn} {\em Twisting data} for a supersymmetric field theory consists of an odd element $Q \in T^W$
and a group homomorphism $\rho:\C^\times \to G_R$ such that
$$\rho(\lambda)(Q)=\lambda Q \quad \forall \lambda\in\C^\times$$
and such that $[Q,Q] = 0$.
\end{defn}
Suppose we have such twisting data, and that we have a theory acted on by $T^W$ with $R$-symmetry group $G_R$.  Then we can, as above, form a family of theories over $\C[\![t]\!]$ by adding $t Q$ to the BRST differential.  We can now, however, use the action of $\C^\times$ on everything to change the grading.  Indeed, this $\C^\times$ action lifts the bi-grading by $\Z \times \Z/2$ to a tri-grading by $\Z \times \Z \times \Z/2$, where the first $\Z$ is the weight under the $\C^\times$-action.  Since $Q$ has weight $1$ under this $\C^\times$ action, $t$ has weight $-1$ and so tri-degree $(-1,1,1)$.  

From this tri-grading we construct a new $\Z \times \Z/2$ grading, by declaring that an element with tri-degree $(a,b,c)$ has new bi-degree $(b+a,c+a)$.  This change of grading respects signs. 

After this change of grading, we see that we have a family of theories over $\C[\![t]\!]$ where $t$ is now of bi-degree $(0,0)$ i.e. it is of cohomological degree $0$ and super degree $0$.  We still have the $\C^\times$-action, and this acts on $t$ by sending $t \to \lambda^{-1} t$.   Therefore this family of theories is independent of $t$, and we can set $t = 1$. 

Thus, our twisting data defines a {\em twisted field theory} with BRST operator $d+Q$ if $d$ is the original BRST operator. For details on the construction, see section 13 in \cite{Cos11b}.

\begin{rem}
Our twisted field theory is a $\C^\times$-equivariant family of theories over $\C$ with BRST operator $d+tQ$. By the Rees construction, this is the same as the data of a filtration on the twisted field theory, whose associated graded is the untwisted theory with a shift of grading. 
It follows that there is a spectral sequence from the cohomology of the observables of the untwisted theory to that of the twisted theory.

One might think that the cohomology of observables of the twisted theory (in the sense above) is a subset of the cohomology of observables of the untwisted theory, because one is looking at the $Q$-closed modulo $Q$-exact observables of the original theory.  This is not really true, however, because this fails to take account of the differential (the BRST operator) on the observables of the untwisted theory. The best that one can say in general is that there is a spectral sequence relating twisted and untwisted observables.   

There are examples (obtained by applying further twists to theories which are already partially twisted) where this spectral sequence degenerates, so that the cohomology of twisted observables has a filtration whose associated graded is the cohomology of untwisted observables. In such cases, twisted and untwisted observables are the ``same size'', and twisted observables are definitely not a subset of untwisted observables (at the level of cohomology). 
\end{rem}

\subsection{Minimally twisted $\CN=1,2,4$ SUSY theories are holomorphic}

We begin with the case $\CN=1$. Choosing an element $Q\in S^+$ is the same as choosing a complex structure on the linear space $\R^4$, with the property that the standard Riemannian metric on $\R^4$ is K\"ahler for this complex structure and that the induced orientation on $\R^4$ is the standard one.  (Elements in $S^-$ give rise to such complex structures which induce the opposite orientation on $\R^4$).

One can see this as follows. Given $Q \in S^+$, the stabilizer $\mathrm{Stab}(Q)\subseteq\Spin(4)\cong\SU(2)\times\SU(2)$ is $\SU(2)_-$, so $Q$ provides a reduction of the structure group to $\SU(2)$. Concretely,  $Q\otimes S^-\subseteq V_{\C}=\C^4$ is the $(0,1)$ part, i.e.~the $-i$ eigenspace of the complex structure, and its complex conjugate is the $(1,0)$ part.

The complexified $R$-symmetry group for $\CN=1$ supersymmetry is $\C^\times$, which acts on supercharges in $S^+$ with weight $1$ and in $S^-$ with weight $-1$.  As we explained earlier, we will use this $R$-symmetry action to change gradings, so that supercharges in $S^+$ have cohomological degree $1$ and those in $S^-$ have cohomological degree $-1$.  
 
After we change the grading in this way, the $\Z$-graded version of the super-translation Lie algebra
$$T^{\CN=1} = S^-[1] \oplus V_{\C} \oplus S^+[-1]$$
acts on the untwisted theory.  If we twist by an element $Q \in S^+$, then the dg Lie algebra $(T^{\CN=1}, [Q,-])$ acts on the $Q$-twisted theory. 

Let $\frac{\partial}{\partial z_i}, \frac{\partial}{\partial \bar{z_i}}$ denote a basis for $V_{\C}$ where we are using the complex structure on $V_\R$ induced by $Q$.  Then the map
$$[Q,-]:S^-\to V_{\C}$$
has image $[Q,S^-]=Q\otimes S^-=V_{\C}^{(0,1)}$, which is the subspace generated by the $\frac{\partial}{\partial \bar{z_i}}$'s. Thus, translations in the $\frac{\partial}{\partial \bar{z_i}}$ directions are homotopically trivial in the twisted theory.

This means that the twisted theory is \emph{holomorphic}.  Let us briefly explain this idea. Recall that the energy-momentum tensor of a field theory arises from the action of the translation group $V_\R$ on the field theory.  One (quite weak) way to say that a field theory is topological is that the energy-momentum tensor is trivial.  This implies, for instance, that correlation functions are independent of position.  Our definition of holomorphic is that the action of $V_\C^{(0,1)}$ is (homotopically) trivial. This will mean that correlation functions are holomorphic functions of position. 

In fact, for $\CN=1,2,4$, any twist by a $Q$ of the form $Q^+\otimes w\in S^+\otimes W$ (a decomposable tensor) produces a holomorphic field theory. Twists by such elements are called {\em minimal twists}.

Examples of such a minimally twisted supersymmetric gauge theory can be obtained by twisting the anti-self-dual $\CN=1,2,4$ supersymmetric gauge theories\footnote{We will refer to these as ``the $\CN=1,2,4$ twisted SUSY gauge theory'' in the rest of these notes.} on $\R^4$. In fact, these twisted field theories arise as cotangent theories, which means that the space of solutions to the equations of motion is described as a $-1$-shifted cotangent bundle:
\begin{description}
\item[$\CN=1$] \hspace{1cm} $T^*[-1](\mbox{holomorphic $G$-bundles})$
\item[$\CN=2$] \hspace{1cm} $T^*[-1](\mbox{holomorphic $G$-bundles + $\psi\in H^0(\g_{P})$})$
\item[$\CN=4$] \hspace{1cm} $T^*[-1](\mbox{holomorphic $G$-bundles + $\psi_1,\psi_2\in H^0(\g_{P})$ s.t.~$[\psi_1,\psi_2]=0$})$
\end{description}
If we work perturbatively (as we do for most of this note), we consider solutions to the equations of motion which lie in a formal neighbourhood of a given solution.  It is possible to glue together the perturbative descriptions over the moduli space of classical solutions, but we do not consider this point in this note.  Here $G$ is a semi-simple algebraic group. We denote by $\g_{P} = P\times_G \g$ the adjoint bundle of Lie algebras associated to $P$.

They admit an explicit description, as derived in \cite{Cos11b}. The fields of these theories can be described in the BV formalism as follows:
\begin{description}
\item[$\CN=1$] The fields are
$$\Omega^{0,\ast}(\C^2, \g)[1]\oplus \Omega^{2,\ast}(\C^2, \g^\vee),$$
and the action on the space of fields is given by
$$\int_{\C^2} \Tr(\beta\wedge (\bar\partial\alpha+\frac12[\alpha,\alpha])),$$
where $\alpha\in\Omega^{0,\ast}(\C^2, \g)[1]$ and $\beta\in\Omega^{2,\ast}(\C^2, \g^\vee)$. This theory is a {\em holomorphic BF theory}, see \cite{Cos13}. It is equivalent to holomorphic Chern-Simons theory on the supermanifold $\C^{2 \mid 1}$. 

\item[$\CN=2$] We get something similar, replacing $\g$ by $\g[\eps]$, where $\eps$ is a square-zero parameter of degree 1.  Thus, the field $\alpha$ is an element $\Omega^{0,\ast}(\C^2,\g[\eps])[1]$ and the field $\beta$ is an element of $\Omega^{2,\ast}(\C^2, (\g[\eps])^\vee)$.  

\item[$\CN=4$] Again, we get something similar, replacing $\g$ by $\g[\eps_1,\eps_2]$, where $\eps_1,\eps_2$ are square-zero parameters of degrees $1$ and $-1$ respectively.  This is equivalent to holomorphic Chern-Simons theory on $\C^{2 \mid 3}$. 
\end{description}

This result from \cite{Cos11b} allows an explicit calculation (at the classical level) of the spaces of observables of these supersymmetric gauge theories.  (We will discuss the structure on observables using the language of factorization algebras shortly).  For instance, for the $\CN=4$ theory, the space of observables supported at the origin in $\C^2$ is 
$$
C^\ast(\g[\![z_1,z_2,\eps_1,\eps_2,\eps_3]\!])
$$
where the $\eps_i$ are three odd parameters. This result was also derived in \cite{ChaYin13}, using different methods.

\section{Factorization algebras in perturbative quantum field theory}

In the book \cite{Cos11}, a definition of a quantum field theory based on Wilsonian effective action and the BV formalism is given. The main result is that we can construct, using renormalization, such perturbative quantum field theories starting from a classical field theory and working term by term in $\hbar$, using obstruction theory.  Let $\E$ be the space of fields of a classical field theory and let $\CO(\E)$ be the functionals on $\E$. If we have a quantization modulo $\hbar$, there may be an obstruction $O_n\in H^1(\CO_{loc}(\E))$ to quantize to the next order. Here $\CO_{loc}(\E)$ denotes the subcomplex of $\CO(\E)$ consisting of local functionals, i.e.~functionals which can be written as sums of integrals over differential operators. If $O_n$ vanishes, we can quantize to the next order, and the possible lifts are a torsor for $H^0(\CO_{loc}(\E))$.
\subsection{Factorization algebras}
In \cite{CosGwi11}, Costello and Gwilliam analyze the structure of observables of a quantum field theory in the language of factorization algebras.  The notion of factorization algebra was introduced in the algebro-geometric context by Beilinson and Drinfeld in \cite{BeiDri04}. The approach used by Costello-Gwilliam is very similar to how observables and the operator product are encoded in Segal's axioms for quantum field theory \cite{Seg99a,Seg04,Seg10}.

\begin{defn}
Let $M$ be a topological space and let $\mathcal{C}$ be a symmetric monoidal category (in examples from field theory $\mathcal{C}$ will be cochain complexes or some variant) . A {\em prefactorization algebra} $\F$ on $M$ (with values in $\mathcal{C}$) consists of the following data.
\begin{enumerate}
\item For every open subset $U\subseteq M$, an object $\F(U)\in \mathrm{Ob}(\mathcal{C})$.
\item If $U_1,\ldots, U_n$ are pairwise disjoint open subsets of an open set $V$, we have a morphism
\begin{center}
 \begin{minipage}{3cm}
 \begin{tikzpicture}[scale=0.25]
 \draw (0,0) circle (5);
 \draw (-1.5,2) circle(1.3) node {$U_1$};
 \draw (-2.2,-1.5) circle (1.5) node {$U_2$};
 \draw (0, -2.5) node {\tiny$\dots$};
 \draw (2.1,-1) circle (1.8) node {$U_n$};
 \draw (2.1, 3) node {$V$};
 \end{tikzpicture}
 \end{minipage}
\hspace{0.5cm} $\rightsquigarrow$ \hspace{0.5cm}
 \begin{minipage}{8cm}
 $$\F(U_1)\otimes\dots\otimes\F(U_n)\longrightarrow\F(V),$$
 \end{minipage}
\end{center}
\end{enumerate}
such that if $U_1\amalg\cdots\amalg U_{n_i}\subseteq V_i$ and $V_1\amalg\cdots\amalg V_k\subseteq W$, the following diagram commutes.
\begin{center}
\hspace{-1cm}
\begin{minipage}{3cm}
{\tiny
\begin{center}
\begin{tikzpicture}[scale=0.3]
\draw (0,0) circle (5);
\draw (2.4, 3.3) node {$W$};
\draw [style=loosely dashed] (-1.7,1.5) circle(2.1);
\draw (-2.2, 2.7) node {$V_1$};
\draw [style=loosely dashed] (2.2,-1.2) circle (2.3);
\draw (3.1, 0) node {$V_2$};
\draw (-2.2, 0.8) circle (0.9) node {$U_1$};
\draw (-0.7, 1.8) circle (0.6) node {$U_2$};
\draw (1.3,-0.5) circle (0.7) node {$U_3$};
\draw (2.8, -2.1) circle (1) node {$U_4$};
\end{tikzpicture}
(for $k=n_1=n_2=2$)
\end{center}
}
 \end{minipage}
 \hspace{0.2cm} $\rightsquigarrow$ \hspace{0.2cm}
\begin{minipage}{8cm}
\begin{tikzcd}[column sep=small]
{\bigotimes}^{k}_{i=1}{\bigotimes}^{n_i}_{j=1}\F(U_j) \arrow{dr} \arrow{rr} &	&{\bigotimes}^k_{i=1}\F(V_i) \arrow{dl}\\
&\F(W)	& 
\end{tikzcd}
\end{minipage}
\end{center}

A {\em factorization algebra} on $M$ is a prefactorization algebra on $M$ which additionally satisfies a gluing condition saying that given an open cover $\{U_i\}$ of $V$ satisfying certain conditions, $\F(V)$ can be recovered from the $\F(U_i)$'s.This glueing condition is analogous to the one for (homotopy) (co-)sheaves.

For the exact gluing condition and more details on the theory of factorization algebras we refer to \cite{CosGwi11} and to Gr\'egory Ginot's notes \cite{Gin13}.
\end{defn}
Although the definition makes sense for an arbitrary topological space, we will only consider factorization algebras on manifolds.

\subsection{Associative algebras are factorization algebras.} 
Actually, associative algebras are a special case of factorization algebras with values in chain complexes. Suppose that we have a factorization algebra $\F$ on $\R$ with the property that  for any interval $(a,b)\subseteq\R$ the map
$$\F((a,b))\overset{\simeq}{\longrightarrow}\F(\R)$$
is a quasi-isomorphism, i.e.~an isomorphism on cohomology\footnote{Such a factorization algebra is called {\em locally constant}.}. Then $\F$ defines an associative algebra (up to homotopy). Let $A=\F(\R)\simeq\F((a,b))$ for any $(a,b)\subseteq\R$. If $(a,b)\amalg (c,d)\subseteq (e,f)$ with $e<a<b<c<d<f$, the factorization algebra structure gives us a map
\begin{center}
\hspace{-3.5cm}\begin{minipage}{12em}
\begin{center}
\begin{tikzpicture}[scale=0.5]
\draw (-5,0) -- (5,0);
\draw (3,0) arc (0:30:0.5);
\draw (3,0) arc (0:-30:0.5);
\draw[gray] (4,0) arc (0:30:0.5);
\draw[gray] (4,0) arc (0:-30:0.5);
\draw (-1,0) arc (0:30:0.5);
\draw (-1,0) arc (0:-30:0.5);

\draw (1,0) arc (0:30:-0.5);
\draw (1,0) arc (0:-30:-0.5);
\draw (-3,0) arc (0:30:-0.5);
\draw (-3,0) arc (0:-30:-0.5);
\draw[gray] (-4,0) arc (0:30:-0.5);
\draw[gray] (-4,0) arc (0:-30:-0.5);

\draw (-3,0.2) node [anchor=south] {$a$};
\draw (-1,0.2) node [anchor=south] {$b$};
\draw (1,0.2) node [anchor=south] {$c$};
\draw (3,0.2) node [anchor=south] {$d$};
\draw[gray] (-4,0.2) node [anchor=south] {$e$};
\draw[gray] (4,0.2) node [anchor=south] {$f$};

\draw (0,-1.5) node {\rotatebox[origin=c]{-90}{$\hookrightarrow$}};

\draw[gray] (-5,-3) -- (5,-3);
\draw[gray] (3,-3) arc (0:30:0.5);
\draw[gray] (3,-3) arc (0:-30:0.5);
\draw (4,-3) arc (0:30:0.5);
\draw (4,-3) arc (0:-30:0.5);
\draw[gray] (-1,-3) arc (0:30:0.5);
\draw[gray] (-1,-3) arc (0:-30:0.5);

\draw[gray] (1,-3) arc (0:30:-0.5);
\draw[gray] (1,-3) arc (0:-30:-0.5);
\draw[gray] (-3,-3) arc (0:30:-0.5);
\draw[gray] (-3,-3) arc (0:-30:-0.5);
\draw (-4,-3) arc (0:30:-0.5);
\draw (-4,-3) arc (0:-30:-0.5);

\draw[gray] (-3,-2.8) node [anchor=south] {$a$};
\draw[gray] (-1,-2.8) node [anchor=south] {$b$};
\draw[gray] (1,-2.8) node [anchor=south] {$c$};
\draw[gray] (3,-2.8) node [anchor=south] {$d$};
\draw (-4,-2.8) node [anchor=south] {$e$};
\draw (4,-2.8) node [anchor=south] {$f$};

\end{tikzpicture}
\end{center}

\end{minipage}
\hspace{1cm} $\rightsquigarrow$\hspace{1cm}
\begin{minipage}{0.2\linewidth}    

\begin{equation*}
\begin{array}{ccccc}
\F((a,b))&\otimes&\F((c,d))&\longrightarrow&\F((e,f))\\
\rotatebox[origin=c]{-90}{$\simeq$}&&\rotatebox[origin=c]{-90}{$\simeq$}&&\rotatebox[origin=c]{-90}{$\simeq$}\\
A&\otimes&A&\overset{m}{\longrightarrow}&A
\end{array}
\end{equation*}
      
\end{minipage}
\end{center}
Conversely, any associative algebra defines a locally constant factorization algebra on $\R$.

\begin{rem}
Such factorization algebras really appear in quantum mechanics. Quantum mechanics is the field theory with fields $\phi \in C^\infty_c(\R)$ and action functional $S(\phi)=\int\phi\Delta\phi$, $\phi\in C^{\infty}_c(\R)$. Then the equations of motion say that $\phi$ is harmonic, i.e.~$\Delta \phi = 0$.  Harmonic functions on $(a,b)$ extend uniquely to harmonic functions on $\R$: this implies that the factorization algebra constructed from this example has the property that the map $\F((a,b)) \to \F(\R)$ is an isomorphism. This example will be explained in more detail in \ref{free field ex}.
\end{rem}

\subsection{The factorization algebra of observables}\label{observables}
It is shown in \cite{CosGwi11} that observables of a (perturbative) quantum field theory in Euclidean signature turn out to have the structure of a factorization algebra with values in the category of cochain complexes of $\C[\![\hbar]\!]$-modules, flat over $\C[\![\hbar]\!]$.  These cochain complexes are built from spaces of smooth functions and distributions on the space-time manifold. Technically, these cochain complexes are endowed with a ``diffeological'' structure, which is something a little weaker than a topology; this reflects their analytical origin.

Observables of a classical field theory also form a factorization algebra. Starting from the quantum observables, the classical observables are
 $$\Obs^{cl}(V)\coloneq\Obs^q(V)/\hbar=\left\{\begin{array}{c}\mbox{functions on the ``derived'' moduli space of}\\ \mbox{solutions to the Euler-Lagrange equations on }V\end{array}\right\}.$$
Taking the derived space of solutions to the Euler-Lagrange equations amounts to one version of the BV classical formalism.  The antifields of the BV formalism correspond to taking the Koszul complex associated to the equations of motion and the ghosts correspond to taking the quotient by the gauge group in a homological way. For more details, see \cite{Cos11}.

The factorization algebra of quantum observables deforms that of classical observables, in that quantum observables are a factorization algebra over $\C[\![\hbar]\!]$ and restrict to classical observables modulo $\hbar$.  To first order, this deformation is closely related to the BV antibracket on the classical observables.

We should emphasize that quantum observables, for a general quantum field theory in Euclidean signature,  \emph{do not} form an associative algebra\footnote{Note that we work in Euclidean signature.  Some axiom systems in Lorentzian signature have an asssociative structure on observables: see Klaus Fredenhagen's lectures in the same volume.}.  Associative algebras arise when one studies factorization algebras on the real line (associated to $1$-dimensional quantum field theories).  The associative product is the operator product of observables in the time direction. For a factorization algebra on a higher-dimensional manifold, there is no specified ``time'' direction which allows one to define the associative product, rather there is a kind of ``product'' for every direction in space-time.  

Furthermore, quantum field theories in dimension larger than one rarely satisfy the locally-constant condition which was satisfied by the observables of quantum mechanics.  The exception to this rule is the observables of a topological field theory.  In this case, however, we find that observables from an $E_n$-algebra (a structure studied by topologists to encode the product in an $n$-fold loop space) rather than simply an associative algebra.  This is a result of Lurie \cite{Lur12}, who shows that there is an equivalence between locally-constant factorization algebras on $\R^n$ and $E_n$-algebras. 

Another point to emphasize is that factorization algebras are only the right language to capture the structure of observables of a QFT in Euclidean signature. In Lorentzian signature, the operator product (at least for massless theories) has singularities on the light-cone and not just on the diagonal, so that we would only expect to be able to define the factorization product for pairs of open subsets which are not just disjoint, but  which can not be connected by a path in the light cone. 

\subsubsection*{More structures on the factorization algebra of observables.}
\hspace{1em}

{\em Translation invariance:} If we additionally have translation invariance on a locally constant factorization algebra on $\R$, we get an associative algebra endowed with an infinitesimal automorphism, i.e.~a derivation. This derivation encodes the Hamiltonian of the field theory. 

{\em Poisson bracket:} The classical observables $\Obs^{cl}(U)$ form a commutative dg algebra. Moreover, we have a Poisson bracket of cohomological degree one\footnote{This means that $\Obs^{cl}$ has the structure of a {\em $P_0$ factorization algebra}, where $P_0$ is the operad describing commutative dg algebras with a Poisson bracket of degree $1$.}, the ``antibracket'' $\{\,,\}$ on $\Obs^{cl}(U)$.

{\em (Weak) Quantization condition:} In deformation quantization, the non-commutative algebra structure to first order must be related to the Poisson bracket. We have a similar condition\footnote{We present here a weak version of the condition.  A stronger version, discussed in \cite{CosGwi11}, is that $\Obs^q$ is a {\em $BD$ factorization algebra}, where $BD$ is the Beilinson-Drinfeld operad. The $BD$ operad is an operad over $\C[\![\hbar]\!]$ deforming the $P_0$ operad:  $BD\otimes_{\C[\![\hbar]\!]}\C \simeq P_0$.} relating the factorization algebras of quantum and classical observables.

The differential $\d$ on $\Obs^q(U)$ should satisfy
\begin{enumerate}
\item Modulo $\hbar$, $\d$ coincides with the differential $\d_0$ on $\Obs^{cl}(U)$.
\item Let 
$$
\d_1 : H^i (\Obs^{cl}(U)) \to H^{i+1}(\Obs^{cl}(U)) 
$$
be the boundary map coming from the exact sequence of complexes
$$
\hbar \Obs^{cl}(U) \longrightarrow \Obs^q (U) \mod \hbar^2 \longrightarrow \Obs^{cl}(U).
$$ 
$\d_1$ lifts to a cochain map of degree $1$ $\Obs^{cl}(U) \to \Obs^{cl}(U)$, which we continue to call $\d_1$.   Then, if we define a bilinear map on $\Obs^{cl}(U)$ by
$$
\{a,b\}^{\d_1} = \d_1(ab) \mp a \d_1 b - (d_1 a) b
$$
we ask that there is a homotopy between $\{a,b\}^{\d_1}$ and the original bracket $\{a,b\}$.  (In particular, these two brackets must coincide at the level of cohomology). 
\end{enumerate}

\subsection{Example: the free scalar field}\label{free field ex}

Let $M$ be a compact Riemannian manifold.  We will consider the field theory where the fields are $\phi \in C^{\infty}(M)$, and the action functional is 
$$
S(\phi) = \int_M \phi \Delta \phi,
$$
where $\Delta$ is the Laplacian on $M$.

\subsubsection{Classical observables}

If $U \subseteq M$ is an open subset, then the space of solutions of the equations of motion on $U$ is the space of harmonic functions on $U$,
$$\{\phi\in C^{\infty}(U)|\Delta\phi=0\}.$$

As discussed above, we consider the \emph{derived} space of solutions of the equations of motion. This is a linear dg manifold, i.e.~a cochain complex. For more details about the derived philosophy, the reader should consult \cite{CosGwi11}. In this simple situation, the derived space of solutions to the free field equations, on an open subset $U \subseteq M$, is the two-term complex
$$\E(U) = \left( C^{\infty}(U) \overset{\Delta}{\longrightarrow} C^{\infty}(U)[-1]\right).$$

The classical observables of a field theory on an open subset $U \subseteq M$ should be functions on the derived space of solutions to the equations of motion on $U$, and thus the symmetric algebra of the dual.\footnote{For free theories, it is enough to consider polynomial functions.} The dual to the two-term complex $\E(U)$ above is the complex
$$
\E_c^\vee(U) = \left( \mathcal{D}_c(U)[1] \overset{\Delta}{\longrightarrow} \mathcal{D}_c(U) \right),
$$ 
where $\mathcal{D}_c(U)$ indicates the space of compactly supported distributions on $U$.

We would like to define $\Obs^{cl}=\CO(\E)=\Sym(\E^\vee)$, but in order to define a Poisson structure on $\Obs^{cl}$, instead we need to use a version of elliptic regularity, which we call the Atiyah-Bott lemma \cite{AtiBot67}. Let
$$\E_c^!(U)=(\cinfty_c(U)[1]\longrightarrow \cinfty_c(U)).$$
The Atiyah-Bott lemma states that the map of cochain complexes
$$\E_c^!(U) \longrightarrow \E_c^\vee(U),$$
given by viewing a compactly supported function as a distribution, is a continuous homotopy equivalence.

Thus, we define our classical observables to be
$$
\Obs^{cl}(U) = {\Sym} \left( \E_c^!(U)\right) = \oplus_n \Sym^n \E_c^!(U).
$$
By $\Sym^n \E_c^!(U)$ we mean the $S_n$-invariants in the complex of compactly supported sections of the bundle $(E^!)^{\boxtimes n}$ on $U^n$. Equivalently, we can view $\Sym^n \E_c^!(U)$ as the symmetric product of the topological vector space $\E_c^!(U)$ using the completed inductive (or bornological) tensor product\footnote{These tensor products both have the property that $\cinfty_c(M) \what{\otimes} \cinfty_c(N) = \cinfty_c(M \times N)$, and similarly for compactly supported smooth sections of a vector bundle on $M$.  The more familiar projective tensor product does not (at least not obviously) have this property. See \cite{Gro52} for a discussion of the inductive tensor product and \cite{KriMic97} for the bornological tensor product.  The reader with no taste for functional analysis should just take the fact that $\E_c(U)^{\otimes n} = \Gamma_c(U, E^{\boxtimes n})$ as a definition of $\E_c(U)^{\otimes n}$.}.

It is clear that classical observables form a prefactorization algebra.  Indeed, $\Obs^{cl}(U)$ is a differential graded commutative algebra.  If $U \subseteq V$, there is a natural algebra homomorphism 
$$i^{U}_V : \Obs^{cl}(U) \to \Obs^{cl}(V),$$ which on generators is just the natural map $C^{\infty}_c(U) \to C^{\infty}_c(V)$ given by extending a continuous compactly supported function on $U$ by zero on $V\setminus U$.

If $U_1,\dots, U_n \subseteq V$ are disjoint open subsets, the prefactorization structure map is the continuous multilinear map
\begin{align*}
\Obs^{cl}(U_1) \times \dots \times \Obs^{cl}(U_n) & \to \Obs^{cl}(V) \\
\alpha_1 \times \dots \times \alpha_n & \mapsto \sum_{i=1}^n i^{U_i}_V \alpha_i \in \Obs^{cl}(V).  
\end{align*}

In dimension one, this is particularly simple.
\begin{lemma}\cite{CosGwi11} 
If $U = (a,b) \subset \R$ is an interval in $\R$, then
\begin{enumerate}
\item For any $x\in(a,b)$, the complex
$$\E_c^!((a,b))=\left(C^{\infty}_c((a,b))[1] \overset{\Delta}{\rightarrow} C^{\infty}_c((a,b))\right)$$
is quasi-isomorphic to $\R^2$ situated in degree 0, i.e.~the cohomology is
$$H^\ast\left(\E_c^!((a,b))\right) = \R^2.$$

\item The algebra of classical observables for the free field has cohomology
$$
H^\ast\left(\Obs^{cl}((a,b))\right) = \R[p,q],
$$ 
the free algebra on two variables.
\end{enumerate}
\end{lemma}

\begin{proof}
We first show that for any $x_0\in(a,b)$, the map 
$$\begin{array}{rcll}
\E((a,b))=&\left(\cinfty((a,b)) \overset{\Delta}{\to} \cinfty((a,b))[-1]\right) & \overset{\pi}{\longrightarrow} &\R^2\\
&(\phi,\psi) & \longmapsto & (\phi(x_0), \phi'(x_0))
\end{array}$$
is a quasi-isomorphism. To show this, consider the inclusion
\begin{align*}
\R^2 &\overset{i}{\longhookrightarrow} \E((a,b))\\
(a,b) & \longmapsto (a+b(x-x_0), 0).
\end{align*}
Then $\pi\circ i = \id|_{\R^2}$ and $i\circ\pi: (\phi,\psi)\mapsto(\phi(x_0)+\phi'(x_0)(x-x_0),0)\in\E((a,b))$. Thus, a homotopy between the identity and $i\circ\pi$ is given by
$$(\phi,\psi)\longmapsto S(\phi,\psi)(x)=\left(\int_{y=a}^x\int_{u=a}^y \psi(u)\,\d u\d y,0\right),$$
as $\id-i\circ\pi=[\Delta, S]$.

This implies that the dual $\E^\vee((a,b))$ also is quasi-isomorphic to $\R^2$ and, by elliptic regularity, $\E^!_c((a,b))\simeq\E^\vee((a,b))\simeq\R^2$.

The second part follows directly from the first and the exactness of $\Sym_\R$.
\end{proof}

\begin{rem}
The quasi-isomorphism $\pi$ from the proof induces the desired quasi-isomorphism $\pi^\vee:(\R^2)^\vee\to\E^\vee$. So,
$$\begin{array}{rcccl}
\pi^\vee(1,0)(\phi,\psi)&=&\phi(x_0)&=&\delta_{x_0}(\phi)=q(\phi,\psi)\\
\pi^\vee(0,1)(\phi,\psi)&=&\phi'(x_0)&=&\delta'_{x_0}(\phi)=p(\phi,\psi),
\end{array}$$
the position and the momentum observables, respectively. Thus, the cohomology of $\E^\vee((a,b))$ is generated by $q$ and $p$.
\end{rem}

Recall that the classical observables are endowed with a Poisson bracket of degree 1. For $\alpha \in C^{\infty}_c(U)$ and $\beta \in C^{\infty}_c(U)[1]$, we have
$$
\{\alpha,\beta\} = \int_U \alpha \beta \,dVol.
$$ 
This extends uniquely to a continuous Poisson bracket on $\Obs^{cl}(U)$.

\subsubsection{Quantizing free field theories}\label{quantization}

Our philosophy is that we should take a $P_0$ factorization algebra $\Obs^{cl}$ (e.g.~the observables of a classical field theory) and deform it into a $BD$ factorization algebra $\Obs^q$. This is a strong version of the quantization condition. For a general (interacting) field theory, with the current state of technology we can only construct a weak quantization as defined in section \ref{observables}. However, in the case of a free field theory, we can show that the quantization satisfies this strong quantization condition.

Now we will construct such a quantization of the classical observables of our free field theory, i.e.~a factorization algebra $\Obs^q$ with the property that
$$\Obs^q(U)=\Obs^{cl}(U)[\hbar]$$
as $\C[\hbar]$-modules and with a differential $\d$ such that
\begin{enumerate}
\item Modulo $\hbar$, $\d$ coincides with the differential on $\Obs^{cl}(U)$,
\item $\d$ satisfies
$$\d(ab)=(\d a)b + (-1)^{|a|} a(\d b)+\hbar\{a,b\},$$
where the multiplication arises from that on $\Obs^{cl}(U)$.
\end{enumerate}

The construction (see also \cite{Gwi12}) starts with a certain graded Heisenberg Lie algebra.  Let 
$$
\mathcal{H}(U) = \left(  \cinfty_c(U) \overset{\Delta}{\longrightarrow} \cinfty_c(U)[-1] \right) \oplus \R\hbar[-1],
$$
where $\R$ is situated in degree $1$.  Let us give $\mathcal{H}(U)$ a Lie bracket by saying that, if $\alpha \in \cinfty_c(U)$ and $\beta \in \cinfty_c(U)[-1]$, then 
$$
[\alpha,\beta] = \hbar \int_U \alpha \beta.
$$

Let 
$$
\Obs^q(U) = C_{-\ast} (\mathcal{H}(U) )
$$
be the Chevalley-Eilenberg Lie algebra chain complex of $\mathcal{H}(U)$ with the grading reversed.  The tensor product that is used to define the Chevalley chain complex is, as before, the completed inductive (or bornological) tensor product of topological vector spaces.  

Thus,
\begin{align*} 
\Obs^q(U) &= ( \Sym^\ast (\mathcal{H}(U)[1]), \d )\\
&= ( \Obs^{cl}(U)[\hbar], \d ) \\
&= \left(  \oplus_n \Gamma_c(U^n, (E^!)^{\boxtimes n} )^{S_n} \right)[\hbar]
\end{align*}
where, in the last line, $E^!$ is the direct sum of the trivial vector bundles in degrees $0$ and $-1$.   

The differential $\d$ is defined by first extending the Lie bracket by
$$[\alpha,\beta\wedge\gamma]=[\alpha,\beta]\wedge\gamma +(-1)^{|\beta|(|\alpha|+1)} \beta\wedge[\alpha,\gamma],$$
and then defining
$$\d(\alpha\wedge\beta)=d\alpha\wedge\beta+(-1)^{|\alpha|}\alpha\wedge d\beta + \hbar[\alpha,\beta].$$

Thus, by definition, $\Obs^q(U)$ is a $BD$-algebra and $\Obs^q$ has the structure of a factorization algebra in $BD$-algebras by extending the natural map $\cinfty(U)\to\cinfty(V)$ by the identity to the central extension.

Finally, one can prove that our construction of the factorization algebra for a free field theory, when restricted to dimension one, reconstructs the Weyl algebra associated to quantum mechanics.

\begin{prop}
Let $\Obs^q$ denote the factorization algebra on $\R$ constructed from the free field theory, as above.  Then,
\begin{enumerate}
\item The cohomology $H^\ast(\Obs^q)$ is locally constant.
\item The corresponding associative algebra is the Weyl algebra, generated by $p,q,\hbar$ with the relation $[p,q] = \hbar$. Classically, $p$ is the observable which sends a field $\phi \in \cinfty(\R)$ to $\phi'(0)$, and $q$ sends $\phi$ to $\phi(0)$. 
\item The fact that this factorization algebra is translation invariant means that the corresponding associative algebra is equipped with a derivation which we call $H$. This derivation is given by the Lie bracket with the Hamiltonian, 
$$
H =  \tfrac{1}{2}\hbar^{-1} [p^2, \,\,\,].
$$
\end{enumerate}
\end{prop}

A proof can be found in the section on quantum mechanics and the Weyl algebra in \cite{CosGwi11}.

\section{Factorization algebras associated to SUSY gauge theories}

In this section, we will mostly consider the $\CN=1$ twisted SUSY gauge theory.
\subsection{Replacing $\C^2$ by a complex surface}

In section \ref{SUSY}, we discussed twistings of SUSY gauge theories. They arose as parts of a theory invariant under some $Q\in S^+$ and gave a holomorphic field theory on $\C^2$.

Recall that  for $\CN=1$, the fields $\E$ of the twisted theory on $\C^2$ are built from a Lie algebra $\g$ with associated elliptic Lie algebra $\CL_{\CN=1}=\E[-1]$,
\begin{equation}\label{L on C^2}
\CL_{\CN=1}=\Omega^{0,\ast}(\C^2,\g)\oplus\Omega^{2,\ast}(\C^2,\g^\vee)[-1]\cong\Omega^{0,\ast}(\C^2,\g\oplus\g^\vee[-1]),
\end{equation}
with differential $\bar{\partial}$, and Lie bracket determined by 
$$[\beta,\beta']=0, \quad [\alpha, \beta]=\ad^*_{\alpha}(\beta)\in\Omega^{2,\ast}(\C^2,\g^\vee)[-1], \quad [\alpha,\alpha']\in\Omega^{0,\ast}(\C^2,\g),$$
for $\alpha,\alpha'\in\Omega^{0,\ast}(\C^2,\g), \beta,\beta'\in\Omega^{2,\ast}(\C^2,\g^\vee)[-1]$.
The invariant pairing on $\CL_{\CN=1}$ is given by
$$\langle\phi\otimes X,\psi\otimes Y\rangle=\int\phi\wedge\psi\,\langle X,Y\rangle_{\g},$$
where $X\in\g$, $Y\in\g^\vee$, and $\phi,\psi\in\Omega^{0,\ast}(\C^2)$.\footnote{This is actually only well-defined for compactly supported sections, but this technical difficulty can be overcome by passing to a quasi-isomorphic chain complex similar to what we did in \ref{free field ex}. See \cite{CosGwi11} for details.}

Our Chern-Simons action functional is
$$S(\chi)=\frac12\langle\chi,\bar{\partial}\chi\rangle+\frac16\langle\chi,[\chi,\chi]\rangle$$
for a general field $\chi$.

We saw at the end of section \ref{SUSY} that this twisted $\CN=1$ theory is the space of solutions to the equations of motion for the cotangent theory to the pointed moduli problem of holomorphic principal $G$-bundles on $\C^2$. This theory makes sense on a general complex surface $X$, and is the {\em cotangent theory to the pointed moduli problem of holomorphic principal $G$-bundles on $X$},
$$T^*[-1]\operatorname{Bun}_G(X).$$
Similarly to before, we find that the moduli space of solutions to the equations of motion is
\begin{equation*}
\left\{(P,\phi)| P \mbox{ a principal $G$-bundle on }X, \phi\in H^0_{\bar{\partial}}(X, K_X\otimes\g_{\ad}^\vee)\right\},
\end{equation*}
where $K_X$ is the canonical bundle on $X$ and $\g=\mathrm{Lie}(G)$. This problem corresponds to the elliptic Lie algebra
$$\CL_{\CN=1}(X)=\Omega^{0,\ast}(X,\g)\oplus \Omega^{0,\ast}(X, \g^\vee\otimes K_X)[-1].$$
This is because the equations of motion for $\CL_{\CN=1}(X)$ give
$$\bar{\partial}\alpha+\frac12[\alpha,\alpha]$$
for $\alpha\in \Omega^{0,1}(X,\g)$, which is the Maurer-Cartan equation, and for $\beta\in\Omega^{2,0}(X,\g^\vee)$,
$$\bar{\partial}_\alpha \beta\coloneq\bar{\partial}\beta+[\alpha,\beta]=0.$$
Thus, the principal bundle corresponds to a Maurer-Cartan element in $\Omega^{0,\ast}(X,\g)$ and $\phi$ to the element $\beta$.

Recall that the classical observables are functions on the (derived) space of solutions to the equations of motion. If our theory is given by the elliptic Lie algebra $\CL$, the solutions to the equations of motion are given by Maurer-Cartan elements of $\CL$, i.e.~$\chi\in\CL$ such that $\bar{\partial}\chi+\frac12[\chi,\chi]=0$. If $U\subseteq\C^2$ is open, then the classical observables on $U$ are Lie algebra cochains of $\CL(U)$,
$$\Obs^{cl}(U)=C^\ast(\CL(U))=\widehat{\Sym}^\ast(\CL(U)^\vee[-1]).$$

In our case, $\CL=\CL_{\CN=1}(X)$ is the semi-direct product, i.e.~the split-zero extension $\h\ltimes M$ of $\h=\Omega^{0,\ast}(\C^2,\g)$ with $M=\Omega^{2,\ast}(\C^2,\g)$, so if $U$ is a ball, we essentially get
$$\Obs^{cl}(U)=C^\ast(\h\ltimes M)=C^\ast\left(\mathrm{Hol}(U)\otimes\g,\widehat{\Sym}^\ast\left((\mathrm{Hol}(U)\,\d z_1\d z_2\otimes\g)^\vee\right)\right),$$
a fancy (derived) version of functions on $\{\phi\in\mathrm{Hol}(U)\,\d z_1\d z_2\otimes\g\}/\mbox{Gauge}$.

\subsection{Quantization}

Recall that by quantization, we mean that we deform the commutative factorization algebra of classical observables to a quantum one. Essentially the differential is deformed by using the BV Laplacian and by replacing the classical action of our field theory by a quantum one which satisfies a renormalized BV quantum master equation. In our case, we find that there is a quantization, and it even is unique:

\begin{thm}
The $\CN=1$ minimally twisted supersymmetric gauge theory on a complex surface $X$ with trivial canonical bundle,  perturbing around any holomorphic $G$-bundle for a simple algebraic group $G$, admits a unique quantization compatible with certain natural symmetries.
\end{thm}
The proof of this is given in \cite{Cos13}. It relies on the renormalization theory from \cite{Cos11} which reduces it to a cohomological calculation. More precisely, \cite{Cos11} tells us that, order by order in $\hbar$, the obstruction to quantizing to the next order lies in $H^1(\CO_{loc})$, and the ambiguity in quantizing is given by $H^0(\CO_{loc})$. Similar techniques can be used to show that the twisted $\CN=2$ and $4$ theories can be quantized on any complex surface $X$.

\begin{rem}
This theorem is a special case of a very general result. Recall that the fields of the twisted $\CN=1$ theory are $T^\ast[-1] \op{Bun}_G(X)$.  Since $X$ has trivial canonical bundle, $\op{Bun}_G(X)$ has a symplectic form already.   The general result is that the cotangent theory to an elliptic moduli problem which is already symplectic has a natural quantization (it's the unique quantization compatible with certain symmetries).
\end{rem}

\subsection{The relation to vertex algebras}

As we saw in the previous section, quantization gives a factorization algebra of quantum observables, which looks like
$$\Obs^q:U\longmapsto \left(C^\ast(\CL(U))[\![\hbar]\!], \d\right),$$
where modulo $\hbar$, $\d$ coincides with the differential $\d_{CE}$ on $C^\ast(\CL(U))$, cf.~\ref{quantization}. In particular, the factorization algebra structure of our minimally twisted $\CN=1,2,4$ theory on $\C^2$ associates a product to each configuration of $k$ balls inside a larger one in $\C^2$. Now consider the situation where the big ball is centered at the origin. By translation invariance of our theory, up to isomorphism $\Obs^q(B_r(z))$ is independent of $z$, and we call this $V_r$. Thus, for every  $p=(z_1,\ldots,z_k)$, a point in the parameter space
$$P(r_1,\ldots,r_k|s)=\{k\mbox{ disjoint balls of radii $r_1,\ldots,r_k$ inside } V_s\},$$
we get a map $m_p:V_{r_1}\otimes\cdots \otimes V_{r_k}\longrightarrow V_s.$
 \begin{center}
 \begin{minipage}{3cm}
 \begin{tikzpicture}[scale=0.3]
 {\tiny
 \draw (0,0) circle (5);
 \draw (2.5,-1) circle (1.8) node {$B_{r_k}\!(z_k)$};
 \draw (-1.5,1.8) circle(1); 
 \draw (-1.5,3.5) node {$B_{r_1}\!(z_1)$};
 \draw (-2.2,-0.8) circle (1.0);
 \draw (-2, -2.4) node {$B_{r_2}\!(z_2)$};
 \draw (-0.3, -1.1) node {\tiny$\cdots$};
 }
 {\small
 \draw (5, 4) node {$B_s\!(0)$};
 }
 \end{tikzpicture}
 
 \end{minipage}
\hspace{0.5cm} $\rightsquigarrow$\hspace{0.5cm}
\begin{minipage}{6cm}
$m_p:V_{r_1}\otimes\cdots\otimes V_{r_k}\longrightarrow V_s$
\end{minipage}
 \end{center}

This map depends smoothly on the parameter, i.e.
$$V_{r_1}\otimes\cdots \otimes V_{r_k}\longrightarrow V_s\otimes C^{\infty}(P(r_1,\ldots,r_k|s)).$$
If the field theory is holomorphic (as in our case), this map lifts to a cochain map of degree 0 compatible with composition
$$V_{r_1}\otimes\cdots \otimes V_{r_k}\longrightarrow V_s\otimes \Omega^{0,\ast}(P(r_1,\ldots,r_k|s)),$$
which leads to a map in cohomology,
\begin{equation}\label{vertex map coho}
H^\ast (V_{r_1})\otimes\cdots \otimes H^\ast(V_{r_k})\longrightarrow H^\ast(V_s)\otimes H^\ast_{\bar\partial}\left(P(r_1,\ldots,r_k|s)\right).
\end{equation}

\begin{rem}
In complex dimension 1, the analogous structure is that of a vertex algebra.
Consider the multiplication arising from the factorization algebra structure for 2 balls inside a larger one. Assume that one is centered at the origin, $B_{r_1}(z_1)=B_{r_1}(0)=:B_{r_1}$.  Let $P'(r_1,r_2 \mid s)$ denote the subspace of $P(r_1,r_2 \mid s)$ where the first ball is centered at the origin.  
 \begin{center}
 \begin{minipage}{3cm}
 \begin{tikzpicture}[scale=0.3]
 {\tiny
 \draw (0,0) circle (4.5);
 \draw (0,0) circle(1.2);
 \draw (-1.7,-0.7) node {$B_{r_1}$};
 \fill (0,0) circle (0.1) node [anchor=north] {$0$};
 \draw (2.2,-2.2) circle (1.2);
 \draw (3.1,-0.4) node {$B_{r_2}\!(z)$};
  \fill (2.2,-2.2) circle (0.1) node [anchor=north] {$z_{}$};
 }
 {\small
 \draw (4.7, 3.7) node {$B_s\!(0)$};
 }
 \end{tikzpicture}
 \end{minipage}
\hspace{0.5cm} $\rightsquigarrow$
\begin{minipage}{8.5cm}
\begin{equation}\label{chain OPE}V_{r_1}\otimes V_{r_2}\longrightarrow V_s\otimes \Omega^{0,\ast}(P(r_1,r_2|s))\end{equation}
\end{minipage}
 \end{center}
Assuming $B_{r_1}$ to be centered at the origin, the center of the second disk $z=z_{2}$ of radius $r_2$ can be anywhere in the annulus of radii $r_1+r_2$ and $s-r_2$. Thus, our parameter space $P'(r_1,r_2 \mid s)$ consists of this annulus. Let us look at the maps analogous to \eqref{vertex map coho}.  The 0th cohomology of $P'(r_1,r_2|s)$ consists of holomorphic functions on the annulus which are just power series in $z$ and $z^{-1}$ which converge on the annulus.  So we get a map of degree zero
$$H^{\ast}(V_{r_1})\otimes H^{\ast}(V_{r_2})\longrightarrow H^\ast(V_s)\otimes \C\{z,z^{-1}\},$$
where $\C\{z,z^{-1}\}$ refers to those Laurent series which converge on the appropriate annulus.  Moreover, in one dimension, all  higher cohomologies vanish, so that's all we get. See \cite{Gwi12} for the worked out example of the $\beta\gamma$ system, and \cite{CosGwi11} for the general theorem that holomorphically translation-invariant factorization algebras in one complex dimension have the structure of a vertex algebra on their cohomology. 
\end{rem}

\begin{rem}
For holomorphic theories in complex dimension $>1$, it's better to use polydiscs instead of discs.  Thus,  $B_{r} \subset \C^2$ should be understood as the product $D_r \times D_r$ of two discs of radius $r$ in $\C$. 
\end{rem}

In our case of a twisted SUSY gauge theory on $\C^2$, we get a two dimensional analog of vertex algebras, i.e.~for $\CN=1,2,4$ and for any semi-simple Lie algebra we get a two dimensional vertex algebra. The problem is to compute this object. Recall from \eqref{vertex map coho} that on cohomology, we get a map of degree zero. As in dimension 1, we restrict to the space $P'(r_1,r_2 \mid s)$ where the first polydisc is centered at the origin. We can identify $P'(r_1,r_2 \mid s)$ with the two complex dimensional analog of an annulus: the complement of one polydisc in another. 

Again, we can compute the Dolbeault cohomology of this space.  By Hartog's theorem, every holomorphic function on $P'$ extends to zero, so as the zeroth cohomology we just get holomorphic functions on the polydisc and therefore a map of degree zero
\begin{equation}\label{OPE1}
H^{\ast}(V_{r_1})\otimes H^{\ast}(V_{r_2})\longrightarrow H^\ast(V_s)\otimes \C\{z_1,z_2\}.
\end{equation} 
(We use the notation $\C\{z_1,z_2\}$ to refer to series which converge on the appropriate polydisc).  

This map extends to $z_1=z_2=0$; this allows us to construct a commutative algebra from the spaces $H^\ast (V_r)$.  Let us assume (as happens in practice) that the map $H^\ast(V_r) \to H^\ast(V_s)$ (with $r < s$) associated to the inclusion of one disc centered at the origin into another is injective.  Then, let 
$$
\mathcal{V} = \bigcup H^\ast(V_r)
$$
and let $F^r \mathcal{V} = H^\ast(V_r)$.  Then, the map \eqref{OPE1} gives $\mathcal{V}$ the structure of a commutative algebra with an increasing filtration by $\R_{> 0}$.  

 The vector fields $\frac{\partial}{\partial z_i}$ act on $H^\ast (V_r)$; they extend to commuting derivations of the commutative algebra $\mathcal{V}$.  The map in equation \eqref{OPE1} (or rather it's completion where we use $\C[\![z_1,z_2]\!]$) is completely encoded by the  filtered commutative algebra $\mathcal{V}$ with its commuting derivations.  

However, there is more structure. We have the following identification:
$$
H^1_{\dbar}( P' ) = z_1^{-1}z_2^{-1}\C\{z_1^{-1},z_2^{-1}\}.
$$
Thus, the first Dolbeault cohomology of the two complex dimensional annulus consists of series in $z_i^{-1}$ with certain convergence properties. So we find that there is a map
\begin{equation}\label{OPE2}
\mu : H^{\ast}(V_{r_1})\otimes H^{\ast}(V_{r_2})\longrightarrow H^\ast(V_s)\otimes z_1^{-1}z_2^{-1}\C\{z_1^{-1},z_2^{-1}\}.
\end{equation}
Thus, at the level of cohomology, \eqref{chain OPE}, resp.~\eqref{OPE1} and \eqref{OPE2}, form an analog of the operator product expansion of a vertex algebra.  

One can check that the structure given by \eqref{OPE2} is a kind of Poisson bracket with respect to the commutative product obtained from \eqref{OPE1}. To define this we need some notation. If $\alpha,\beta \in H^\ast(V_r)$, then $\mu(\alpha,\beta)$ is a class in $H^1_{\dbar}(P') \otimes \mathcal{V}$. (Recall that $P' \subset \C^2$ is an open subset which is the complement of one polydisc in another.)  Thus, $P'$ retracts onto a $3$-sphere $S^3 \subset P'$.  Then, for every $f \in \C[z_1,z_2]$, one can define a bracket by
$$
\{\alpha, \beta\}_f = \int_{S^3} \mu (\alpha,\beta) f
\d z_1 \d z_2 \in \mathcal{V}.
$$
This makes sense, as $\mu(\alpha,\beta) \d z_1 \d z_2 f$ is a closed $3$-form on $P'$, with coefficients in $\mathcal{V}$.  The integral only depends on the homology class of the sphere $S^3 \subset P'$, which we choose to be the fundamental class. 
 
This bracket gives a map
$$
\{-,-\}_f : H^\ast(V_{r_1}) \otimes H^\ast(V_{r_2}) \to \mathcal{V}
$$
and extends to a map 
$$
\{-,-\}_f: \mathcal{V} \otimes \mathcal{V} \to \mathcal{V}.
$$

One can check that $\{-,-\}_f$ is a derivation in the first factor for the commutative product on $\mathcal{V}$, and satisfies an identity similar to the Jacobi identity.  Let us explain the Jacobi identity we find. If $g \in \C[z_1,z_2]$ let us use the notation
$$
g(z_1+w_1, z_2 + w_2) = \sum g'(z_1,z_2) g''(w_1,w_2).
$$
That is, $\C[z_1,z_2]$ is a Hopf algebra, with coproduct coming from addition on the plane $\C^2$, and we are using the Sweedler notation to write the coproduct $\delta (g)$ of an element $g$ as $\delta (g) = \sum g' \otimes g''$.

Then, the analog of the Jacobi identity in our situation is the following: 
$$
\{\{\alpha,\beta\}_f, \gamma \}_g = \{\{\alpha,\gamma\}_g,\beta\}_f + \sum \{\alpha, \{\beta,\gamma\}_{g'}\}_{f g''}.
$$
Note that in the case $f = g = 1$, this is the usual Jacobi identity. 

All these relations follow from the axioms of a holomorphically translation-invariant factorization algebra using Stoke's theorem.  

We have not presented all details of the structure of a higher-dimensional cohomological vertex algebra (i.e. the structure present on the cohomology of a holomorphically-translation invariant factorization algebra).  Hopefully this will be developed in full elsewhere. The interested reader might consider working out and writing down the entire structure, including all relations satisfied by the Poisson brackets described above.

\begin{rem} There is a similar story for topological field theories on $\R^k$.  There, one finds that the $2$-point operator product is, at the level of cohomology, a map
$$
H^\ast (V) \otimes H^\ast(V) \to H^\ast(V) \otimes H_{dR}^\ast ( \R^k \setminus \{0\}).
$$ 
Here $V$ is the complex $\Obs^q(D)$ for any disc $D$ in $\R^k$. 
On the right hand side of this expression we find the de Rham cohomology of a thickened sphere in $\R^k$, whereas in the holomorphic case we found the Dolbeault cohomology of a simliar region. 

At the cochain level, this operator product gives the complex $V$ the structure of an $E_k$-algebra.  If $k > 1$, then the class in $H^0 (\R^k \setminus \{0\})$ gives $H^\ast(V)$ the structure of a commutative algebra, and the class in $H^{k-1} (\R^k \setminus \{0\})$ gives $H^\ast(V)$ a Poisson bracket of cohomological degree $1-k$.  

The holomorphic situation is very analogous: in dimension $k > 1$, we have a commutative algebra with an infinite family of compatible Poisson brackets. 

Examples of $2$-dimensional topological field theories are given by topological twists of $2$-dimensional supersymmetric gauge theories.  For example, the $B$-twist of the $2$-dimensional $\CN=(2,2)$ gauge theory gives the theory described by the elliptic Lie algebra 
$$(\Omega^\ast(\R^2, \g[\eps]),\d_{dR}, [\,,]),$$
where $\eps$ is a square-zero parameter of cohomological degree $1$.  This is entirely analogous to the fact that the $2$-complex dimensional theory arises as a twist of the $\CN=1$ gauge theory; the only difference is that the Dolbeault complex on $\C^2$ has been replaced by the de Rham complex on $\R^2$.  

Further examples arise from topological $\sigma$-models, such as topological quantum mechanics \cite{GraGwi11}, the Poisson $\sigma$-model \cite{CatFel01, Kon97}, and the $B$-model.  (At the perturbative level, the factorization algebra associated to the $A$-model is uninteresting). 

Summarizing, we get the following table
\begin{center}
\renewcommand{\arraystretch}{2}
\begin{tabular}{l|c|p{8.5cm}|p{4cm}}
{\bf Space} & {\bf $Q$} & {\bf Structures on the cohomology of observables} & {\bf Examples/References}\\
\hline
$\R$ & $\d_{dR}$ & associative product  & topological quantum mechanics \cite{GraGwi11}\\
$\R^2$ & $\d_{dR}$ & commutative product \& degree -1 bracket & Poisson-Sigma model \cite{CatFel01}, topologically twisted $\CN=(2,2)$-gauge theory, the $B$-model\\
$\C$ & $\bar{\partial}$ &  holomorphic analog of a vertex algebra & Minimal twists of $2d$ SUSY field theories, for example \cite{Cos10a}, \cite{Cos11b}, $\beta\gamma$ system, \cite{Gwi12}\\
$\C^2$ & $\bar{\partial}$ & commutative product, $2$ commuting derivations, and family of Poisson brackets of degree $1$ parametrized by $f \in \C[z_1,z_2]$. & $\CN=1,2,4$ minimally twisted SUSY gauge theories \cite{Cos11b,Cos13}\\
\end{tabular}
\end{center}
\end{rem}

One can explicitly compute the map $\mu$ in \eqref{OPE2} for the simplest case.
\begin{ex}
Consider the abelian $\CN=1$ gauge theory, i.e.~$\g=\C$. Then
$$
H^\ast(V_r) = \mathscr{O}( \op{Hol} (B_r) \oplus \op{Hol}(B_r)[-1]),$$
where $\mathscr{O}$ indicates the algebra of formal power series. We will use $\phi, \psi$ to denote elements of the two copies of $\op{Hol}(B_r)$: $\phi$ is of degree $0$ and $\psi$ is of degree $1$.  Let $\alpha,\beta$ be the observables defined by
\begin{align*}
\alpha(\phi,\psi) &= \phi(0) \\
\beta(\phi,\psi) &= \psi(0). 
\end{align*}

Then, one finds that the commutative product does not change, but that the map $\mu$ in \eqref{OPE2} is given by the formula
$$\mu(\phi,\psi)=z_1^{-1}z_2^{-1}\hbar c$$
for a certain constant $c$. 
\end{ex}

\subsection{A deformation of the theory.}
Finally, the $\CN=1$ theory has a deformation which is holomorphic in two real dimensions, and topological in the two other real dimensions. This gives a new relation between the $\CN=1$ gauge theory and the Yangian (see \cite{Cos13} for a detailed discussion of this story). 

We deform the action functional on the space of fields $\Omega^{0,\ast}(\C^2,\g[1])\oplus\Omega^{2,\ast}(\C^2,\g^\vee)$ to
$$S^{new}(\alpha,\beta)=S^{old}(\alpha,\beta)+\int\alpha\,\d z_1\wedge\partial\alpha,$$
for $\alpha\in \Omega^{0,\ast}(\C^2,\g[1]), \beta\in\Omega^{2,\ast}(\C^2,\g^\vee)$. Note that this is not invariant under $SL(2,\C)$ anymore.  The moduli space of solutions to the equations of motion of this deformed theory turns out to be holomorphic $G$-bundles on $\C^2$ along with a compatible flat holomorphic connection in the $z_2$ direction.

More generally, such a deformation can be defined on any complex surface with a closed 1-form; in the previous case, we took the complex surface $\C^2$ with the closed 1-form $\d z_1$. As another example, consider the complex surface $\C^*\times \C$ with the closed 1-form $\frac{\d z_1}{z_1}$. The deformed theory on this complex surface can be projected down to $\R_{>0}\times \C$, where we find a theory we could call ``Chern-Simons theory for the loop group''. 

The quantum observables of the deformed theory on the surface $\C^2$ on a formal disc are
$$C^\ast(\g[\![z_1]\!])[\![\hbar]\!]=\Sym^\ast(\g[\![z_1]\!]^\vee)[\![\hbar]\!].$$
We would like to relate this to the Yangian of the Lie algebra $\g$, which is a quantization of the Hopf algebra $U(\g[\![z]\!])$. Here, we have cochains of $\g[\![z]\!]$ instead of the universal enveloping algebra. These are related by Koszul duality\footnote{There are some delicate issues with this Koszul duality statement: to make it work, we need to treat $C^\ast(\g)$ as a filtered commutative dga. Details are given in \cite{Cos13}.}: For a Lie algebra $\mathfrak{h}$,
$$\C\otimes_{C^\ast(\mathfrak{h})}^{\mathbb{L}}\C=(U(\mathfrak{h}))^\vee.$$

This deformed field theory is topological in the second complex direction, so, by fixing a formal disk in the other coordinate, it gives a locally constant factorization algebra on $\R^2$. Locally constant factorization algebras on $\R^2$ are (by a theorem of Lurie \cite{Lur12}) the same as $E_2$-algebras.  A theorem of Dunn (proved in \cite{Lur12} in the context we need) says that $E_2$-algebras are the same as $E_1$-algebras in $E_1$-algebras. The way to view an $E_2$-algebra as an $E_1$-algebra in $E_1$-algebras is by considering sections of the associated locally constant factorization algebra on an open square (or a strip). Then, the factorization algebra structure gives us two products, namely by including two squares into a third (horizontally) next to each other, or by including them (vertically) above each other. Now we can apply a version of Koszul duality to turn the second $E_1$-algebra structure of the $E_2$-algebra into that of a co-$E_1$-algebra, so we get an $E_1$-algebra in co-$E_1$-algebras, i.e.~a bialgebra, which, in fact, is a Hopf algebra. This gives us a Hopf algebra deforming $U(\g[\![z]\!])$: in \cite{Cos13} it is shown that this Hopf algebra is the Yangian. 

At least heuristically, the (partial) Koszul duality geometrically amounts to taking only the fields which vanish on the top and bottom of the strip. Thus, we consider observables on
$D_{strip}=D_z\times \strip$, where $D_z$ is a (formal) disk with coordinate $z$. Note that the observables on $D_{strip}$ are $\Sym^\ast(\g[\![z]\!]^\vee)$. Inclusion of strips vertically and horizontally give us two operations on the (partial) Koszul dual. 
\begin{enumerate}
\item horizontally: The horizontal inclusion of strips 
\begin{tikzpicture}[scale=0.6]
\draw[gray] (0,0) -- (3.5, 0);
\draw[gray] (0,0.4) -- (3.5, 0.4);
\fill[pattern color=gray, pattern=north west lines] (0,0) rectangle (3.5,0.4);

\draw (0.5,0) -- (1.5, 0);
\draw (0.5,0.4) -- (1.5, 0.4);
\fill[pattern=north east lines] (0.5,0) rectangle (1.5,0.4);

\draw (2,0) -- (3, 0);
\draw (2,0.4) -- (3, 0.4);
\fill[pattern=north east lines] (2,0) rectangle (3,0.4);
\end{tikzpicture}
gives an inclusion map $$D_z\times\strip \amalg D_z\times\strip \longrightarrow D_z\times\strip.$$
The factorization algebra structure thus gives an associative product on observables. Essentially this product is given by restricting fields from the large strip to the small strips inside, which gives a coalgebra structure on the fields. Then taking the dual when passing to the observables we get the algebra structure.
\item vertically: Now consider the vertical inclusion of strips
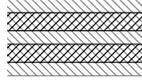
\begin{center}
\begin{tikzpicture}[scale=0.6]
\draw[gray] (0,0) -- (3, 0);
\draw[gray] (0,1.7) -- (3, 1.7);
\fill[pattern color=gray, pattern=north west lines] (0,0) rectangle (3,1.7);

\draw (0,0.3) -- (3, 0.3);
\draw (0,0.7) -- (3, 0.7);
\fill[pattern=north east lines] (0,0.3) rectangle (3,0.7);

\draw (0,1) -- (3, 1);
\draw (0,1.4) -- (3, 1.4);
\fill[pattern=north east lines] (0,1) rectangle (3,1.4);
\end{tikzpicture}
\end{center}
In this case, the condition that the fields vanish on the top and bottom of the strip give us the (additional) coalgebra structure. We can't restrict fields to the smaller strips because of the boundary condition, but instead, we can extend fields by 0 in between the strips and thus get an algebra structure on fields. Here one has to be careful about taking the dual to get the coalgebra structure on the observables, but this can be done in this case.
\end{enumerate}
Moreover, these operations are compatible and give $\Obs^q(D_{strip})=\Sym(\g[\![z]\!]^\vee)[\![\hbar]\!]$ the structure of a Hopf algebra.
\begin{thm}
The Hopf algebra $\Obs^q(D_{strip})$ obtained in this way is dual to the Yangian Hopf algebra.
\end{thm}

So far we have seen how the operator product in the topological direction is encoded by the Yangian Hopf algebra.  In \cite{Cos13} it is shown that the operator product in the holomorphic direction gives rise to a monoidal OPE functor
\begin{equation*}
F_{OPE}: \op{Fin}(Y(\g)) \times \op{Fin}(Y(\g)) \to \op{Fin}(Y(\g))(\!(\lambda)\!) ) \tag{$\dagger$} 
\end{equation*}
which is encoded by the $R$-matrix of the Yangian. Here, $\op{Fin}(Y(\g))$ refers to the monoidal category of finite-rank modules over the Yangian (we use the version of the Yangian which quantizes $\g[\![z]\!]$).  This OPE functor should be thought of as a relative of the OPE in the theory of vertex algebras, except that it takes place in the world of monoidal categories rather than that of vector spaces.

As explained in \cite{Cos13,Cos13b} these two results allows one to calculate exactly expectation values of certain Wilson operators in the twisted, deformed $\CN=1$ supersymmetric gauge theory.  The answer is expressed in terms of the integrable lattice model constructed from the $R$-matrix of the Yangian. 

\subsection{Other Riemann surfaces}
The field theory we are considering can be put on $\Sigma \times \R^2$ where $\Sigma$ is any Riemann surface equipped with a nowhere-vanishing holomorphic $1$-form.  This construction will associate an $E_2$-algebra to any such Riemann surface. In this subsection we will briefly discuss some conjectures about these $E_2$-algebras and related objects. 

If we take our Riemann surface to be $\C^\times$, equipped with the holomorphic volume form $\d z / z$, we find an $E_2$-algebra to which we can apply the Koszul duality considerations above to produce a Hopf algebra. 
\begin{conjecture}
The Hopf algebra Koszul dual to the $E_2$-algebra $\Obs^q(\C^\times \times \R^2)$ is dual to the quantum loop algebra $U_\hbar(\g\{z,z^{-1}\})$.
\end{conjecture}
Note that there are some subtle issues which we have not addressed to do with which completion of $\g[z,z^{-1}]$ one should use to get precisely the Koszul dual of the $E_2$-algebra associated to the cylinder $\C^\times$.  

Next, let us discuss the case of an elliptic curve $E$ equipped with a holomorphic volume form.  Modules for the $E_2$-algebra associated to an elliptic curve form a monoidal category which deforms the category of sheaves on the formal neighbourhood of the trivial bundle in the stack $\op{Bun}_G(E)$ of $G$-bundles on $E$. It is natural to conjecture that  this monoidal category should globalize to a monoidal deformation of the category of quasi-coherent sheaves on $\op{Bun}_G(E)$.    (Quantizations of categories of sheaves like this are considered in \cite{PanToeVaq11}, where it is shown that the stack $\op{Bun}_G(E)$ has a $1$-shifted symplectic form). 

Let us denote this putative quantization by $\op{QC}^q(\op{Bun}_G(E))$.  We conjecture that this monoidal category forms part of a kind of categorified two-dimensional field theory, so that there are analogs of familiar objects such as correlation functions.

More precisely, we conjecture the following.
\begin{enumerate}
\item For every collection of distinct points $p_1,\dots,p_n \in E$, there is a monoidal ``correlation functor'' 
$$
\op{Fin}(Y(\g)) \times \dots \times \op{Fin}(Y(\g)) \to \op{QC}^q(\op{Bun}_G(E)).
$$
Here, $\op{Fin}(Y(\g))$ refers to the monoidal category of finite-rank modules over the Yangian. 

If $M_1,\dots,M_n$ are modules for the Yangian, we denote by
$$
\left\langle M_1(p_1), \dots, M_n(p_n)\right\rangle \in \op{QC}^q(\op{Bun}_G(E))$$
the image of $M_1 \times\dots  M_n$ under the correlation functor.
\item
These correlation functors should quantize the pull-back map on sheaves associated to the map of stacks
$$
\op{Bun}_G(E) \to \left( B G[\![z]\!] \right)^n
$$
obtained by restricting a $G$-bundle on $E$ to the formal neighbourhood of the points $p_i$, where each such formal neighbourhood is equipped with its canonical coordinate arising from the $1$-form on $E$.
\item
All this data should vary algebraically with the positions of the points $p_i$ as well as over the moduli of elliptic curves equipped with a non-zero $1$-form. 
\item
The correlation functors should have a compatibility with the OPE-functor ($\dagger$) in the same way that ordinary correlation functions of a conformal field theory are compatible with the OPE.  For example, if $p, p+\lambda$ are two points in $E$ where $\lambda$ is a formal parameter, and $M,N$ are two modules for the Yangian, we expect that there is a monoidal natural isomorphism
$$
\left\langle M(p), N(p+\lambda)\right\rangle \cong \left\langle F_{OPE} (M,N) (p)  \right\rangle \in \op{QC}^q(\op{Bun}_G(E))(\!(\lambda)\!).
$$
\end{enumerate}
(In the last line, by $\op{QC}^q(\op{Bun}_G(E))(\!(\lambda)\!)$ we mean an appropriate category of $\C(\!(\lambda)\!)$-modules in $\op{QC}^q(\op{Bun}_G(E))$).

Note that if we replace $\op{Bun}_G(E)$ by its formal completion near the trivial bundle, all of this follows from the results of \cite{Cos13}.  Globalizing is the challenge. 

It is natural to speculate that there is some relationship between the desired quantization of $\op{Bun}_G(E)$ and elliptic quantum groups, but this is currently unclear.

In a similar way, for a general surface $\Sigma$ equipped with a nowhere-vanishing holomorphic $1$-form, one can also speculate the $E_2$-algebra of observables of our theory on $\Sigma$ times a disc is related to the quasi-Hopf algebras constructed by Enriquez and Rubtzov \cite{EnrRub99}. 

Kapustin \cite{Kap06} has shown that any four-dimensional $\mathcal{N}=2$ theory admits such a twist.  We hope that there is a similarly rich, and largely unexplored,  mathematical story describing such theories.

\vfill
\pagebreak
\appendix

\section{Moduli problems and field theories}\label{appendixA}

Throughout this text, field theories are described in terms of elliptic moduli problems which in turn are encoded as elliptic $L_\infty$-algebras. These terms and their relations are constantly used. However, we only defined SUSY field theories in an informal way, so we will give some ideas and definitions here. For the full definitions and detailed explanations, see \cite{Cos11b}.

Let $M$ be a manifold. The ideal definition of a classical field theory would be to say that a classical field theory on $M$ is a sheaf of derived stacks (of critical loci, the derived spaces of solutions to the equations of motion) on $M$ equipped with a Poisson bracket of degree one (coming from the BV formalism). To simplify things, we make two observations.

\begin{enumerate}
\item If $X$ is a derived stack and $x\in X$, then $T_xX[-1]$ has an $L_\infty$ structure, and this completely describes the formal neighborhood of $x$, \cite{KonSoi, Hin01, Toe06, Lur11}. Thus, near a given section, a sheaf of derived stacks can be described by a sheaf of $L_\infty$-algebras.
\item If $X$ is a derived stack which is $n$-symplectic in the sense of \cite{PanToeVaq11} , then $T_xX$ has an anti-symmetric pairing (of degree $n$), so $T_xX[-1]$ has a symmetric pairing (of degree $n-2$).  One can show that the $L_\infty$-structure on $T_x X[-1]$ can be chosen so that the pairing is invariant.  More precisely, one can prove  a formal Darboux theorem showing that formal symplectic derived stacks are the same as $L_\infty$-algebras with an invariant pairing. 
\end{enumerate}

From these observations it makes sense to define a perturbative classical field theory (perturbing around a given solution to the equations of motion) to be a sheaf of $L_\infty$-algebras with some sort of an invariant pairing, which we will define below. Moreover, we are interested in the situation where the equations of motion (or equivalently our moduli problem) are described by a system of elliptic partial differential equations, which lead to the following notion.
\begin{defn}
An {\em elliptic $L_\infty$-algebra $\CL$} on $M$ consists of
\begin{itemize}
\item a graded vector bundle $L$ on $M$, whose space of sections in $\CL$,
\item a differential operator $\d:\CL\to\CL$ of cohomological degree 1 and square 0, which makes $\CL$ into an elliptic complex,
\item a collection of polydifferential operators $l_n:\CL^{\otimes n}\to\CL$ which are alternating, of cohomological degree $2-n$, and which give $\CL$ the structure of an $L_\infty$-algebra.
\end{itemize}

An {\em invariant pairing of degree $k$} on an elliptic $L_\infty$-algebra $\CL$ is an isomorphism of $\CL$-modules
$$\CL\cong\CL^![-k],$$
which is symmetric, where  $\CL^!(U)=\Gamma(U, L^\vee\otimes Dens_M)$. 
\end{defn}
\begin{remark}
Note that the sheaf $\CL^!$ is homotopy equivalent to the continuous Verdier dual, which assigns to $U$ the linear dual of $\L_c(U)$.   
\end{remark}
Such an invariant pairing yields an invariant pairing on the space $\L_c(U)$ for every open $U$ in $M$.  The fact that the pairing on $\L_c(U)$ is invariant follows from the fact that the map $\CL \to \CL^![-k]$ is an isomorphism of $\CL$-modules.

From deformation theory, we know that there is an equivalence of $(\infty,1)$-categories between the category of differential graded Lie algebras and the category of formal pointed derived moduli problems (see \cite{Lur11,Hin01,KonSoi}). Here pointed means that we are deforming a given solution to the equations of motion. Thus, the following definitions make sense.

\begin{defn}
A {\em formal pointed elliptic moduli problem with a symplectic form of cohomological degree $k$ on $M$} is an elliptic $L_\infty$-algebra on $M$ with an invariant pairing of cohomological degree $k-2$.
\end{defn}
\begin{defn}
A {\em perturbative classical field theory} on $M$ is a formal pointed elliptic moduli problem on $M$ with a symplectic form of cohomological degree -1. The {\em space of fields $\E$} of a classical field theory arises as a shift of the $L_\infty$-algebra encoding the theory, $\E=\CL[1].$
\end{defn}

The field theories we consider in this text all arise as cotangent theories.
\begin{defn}
Let $\CL$ be an elliptic $L_{\infty}$-algebra on $M$ corresponding to a sheaf of formal moduli problems $\CM_\CL$ on $M$. Then the {\em cotangent field theory} associated to $\CL$ is the classical field theory $\CL\oplus\CL^![-3]$ (with its obvious pairing). Its moduli problem is denoted by $T^*[-1]\CM_{\CL}$.
\end{defn}

\section{Supersymmetry}

In supersymmetry, we have two gradings: one by $\Z/2\Z$ ({\em=fermionic grading}), and one by $\Z$ ({\em =cohomological grading, ``ghost number''}). So one extends the definitions from appendix \ref{appendixA} to this bi-graded (=super) setting.

In this super-setting, we want all algebraic structures to preserve the fermion degree and have the same cohomological degree as in the ordinary setting. Thus, the differential of a {\em super cochain complex} is of degree (0,1) and the structure maps of a {\em super $L_\infty$-algebra $L$}, $l_n: L^{\otimes n}\to L$, are of bidegree $(0,2-n)$, satisfying the same relations as in the ordinary case. The other notions from appendix \ref{appendixA} carry over similarly.

\begin{defn}
A {\em perturbative classical field theory with fermions} on $M$ is a super elliptic $L_\infty$-algebra $\CL$ on $M$ with an invariant pairing of bi-degree (0,-3), i.e.~of cohomological degree -3 and fermionic degree 0.
\end{defn}

\begin{defn}
A {\em formal pointed super elliptic moduli problem with a symplectic form of cohomological degree $k$ on $M$} is a super elliptic $L_\infty$-algebra on $M$ with an invariant pairing of bi-degree $(0,k-2)$.
\end{defn}

Now we can encode supersymmetry.
\begin{defn}
A {\em field theory on $\R^4$ with $\CN=k$ supersymmetries} is a $\Spin(4)\ltimes\R^4$-invariant super elliptic moduli problem $\CM$ defined over $\C$ with a symplectic form of cohomological degree -1; together with an extension of the action of the complexified Euclidean Lie algebra $\mathfrak{so}(4,\C)\ltimes V_{\C}$ to an action of the complexified super-Euclidean Lie algebra $\mathfrak{so}(4,\C)\ltimes T^{\CN=k}$.

Given any complex Lie subgroup $G\subseteq \GL(k,\C)$, we say that such a supersymmetric field theory has {\em $R$-symmetry group $G$} if the group $G$ acts on the theory in a way covering the trivial action on space-time $\R^4$, and compatible with the action of $G$ on $T^{\CN=k}$.
\end{defn}

\vfill
\pagebreak

\def\cprime{$'$}

\end{document}